\documentclass{amsart}
\usepackage{amsmath,amsthm,amscd,amssymb,latexsym,amsfonts,wasysym,xypic}
\usepackage{fancyhdr}
\usepackage{mathrsfs}
\usepackage{hyperref}

\def\a{\alpha}
\def\b{\beta}

\def\d{\delta}

\def\e{\epsilon}
\def\f{\frac}                          

\def\G{\Gamma}
\def\k{\kappa}
\def\lb\{{\left\{}                     
\def\la{\lambda}
\def\La{\Lambda}
\def\lla{\longleftarrow}               
\def\lm{\limits}                       
\def\lra{\longrightarrow}              
\def\dllra{\Longleftrightarrow}        
\def\llra{\longleftrightarrow}         
\def\n{\nabla}
\def\ngth{\negthickspace}              
\def\ngtn{\negthinspace}              
\def\ola{\overleftarrow}               
\def\Om{\Omega}
\def\om{\omega}
\def\op{\oplus}
\def\oper{\operatorname}             
\def\oplm{\operatornamewithlimits}   
\def\ora{\overrightarrow}            
\def\ov{\overline}                   
\def\ova{\overarrow}                 
\def\ox{\otimes}                     
\def\p{\partial}                     
\def\rb\}{\right\}}                  
\def\s{\sigma}
\def\sbq{\subseteq}                  
\def\spq{\supseteq}                  
\def\sqp{\sqsupset}                  
\def\supth{{\text{th}}}              
\def\T{\Theta}
\def\th{\theta}
\def\tkap{\thickapprox}              
\def\tl{\tilde}
\def\tril{\triangleleft}             
\def\thra{\twoheadrightarrow}        
\def\un{\underline}                  
\def\ups{\upsilon}
\def\vp{\varphi}
\def\vt{\vartheta}
\def\wh{\widehat}                    
\def\wt{\widetilde}                  
\def\x{\times}                       
\def\z{\zeta}
\def\({\left(}
\def\){\right)}
\def\[{\left[}
\def\]{\right]}
\def\<{\left<}
\def\>{\right>}

\def\ra{\rightarrow}

\newcommand{\thom}{\widetilde{\hom}}

\def\tec{Teichm\"uller\ }
\def\sconr{\hbox{\medspace\vrule width 0.4pt height 4.7pt depth
0.4pt \vrule width 5pt height 0pt depth 0.4pt\medspace}}
\def\SA{\mathcal A}
\def\SB{\mathcal B}
\def\SC{\mathcal C}
\def\SD{\mathcal D}
\def\SE{\mathcal E}
\def\SF{\mathcal F}
\def\SG{\mathcal G}
\def\SH{\mathcal H}
\def\SI{\mathcal I}
\def\SJ{\mathcal J}
\def\SK{\mathcal K}
\def\SL{\mathcal L}
\def\SM{\mathcal M}
\def\SN{\mathcal N}
\def\SO{\mathcal O}
\def\SP{\mathcal P}
\def\SQ{\mathcal Q}
\def\SR{\mathcal R}
\def\SS{\mathcal S}
\def\ST{\mathcal T}
\def\SU{\mathcal U}
\def\SV{\mathcal V}
\def\SW{\mathcal W}
\def\SX{\mathcal X}
\def\SY{\mathcal Y}
\def\SZ{\mathcal Z}

\newcommand{\BA}{\ensuremath{\mathbf A}}
\newcommand{\BB}{\ensuremath{\mathbf B}}
\newcommand{\BC}{\ensuremath{\mathbf C}}
\newcommand{\BD}{\ensuremath{\mathbf D}}
\newcommand{\BE}{\ensuremath{\mathbf E}}
\newcommand{\BF}{\ensuremath{\mathbf F}}
\newcommand{\BG}{\ensuremath{\mathbf G}}
\newcommand{\BH}{\ensuremath{\mathbf H}}
\newcommand{\BI}{\ensuremath{\mathbf I}}
\newcommand{\BJ}{\ensuremath{\mathbf J}}
\newcommand{\BK}{\ensuremath{\mathbf K}}
\newcommand{\BL}{\ensuremath{\mathbf L}}
\newcommand{\BM}{\ensuremath{\mathbf M}}
\newcommand{\BN}{\ensuremath{\mathbf N}}
\newcommand{\BO}{\ensuremath{\mathbf O}}
\newcommand{\BP}{\ensuremath{\mathbf P}}
\newcommand{\BQ}{\ensuremath{\mathbf Q}}
\newcommand{\BR}{\ensuremath{\mathbf R}}
\newcommand{\BS}{\ensuremath{\mathbf S}}
\newcommand{\BT}{\ensuremath{\mathbf T}}
\newcommand{\BU}{\ensuremath{\mathbf U}}
\newcommand{\BV}{\ensuremath{\mathbf V}}
\newcommand{\BW}{\ensuremath{\mathbf W}}
\newcommand{\BX}{\ensuremath{\mathbf X}}
\newcommand{\BY}{\ensuremath{\mathbf Y}}
\newcommand{\BZ}{\ensuremath{\mathbf Z}}

\def\bba{{\mathbb A}}
\def\bbb{{\mathbb B}}
\def\bbc{{\mathbb C}}
\def\bbd{{\mathbb D}}
\def\bbe{{\mathbb E}}
\def\bbf{{\mathbb F}}
\def\bbg{{\mathbb G}}
\def\bbh{{\mathbb H}}
\def\bbi{{\mathbb I}}
\def\bbj{{\mathbb J}}
\def\bbk{{\mathbb K}}
\def\bbl{{\mathbb L}}
\def\bbm{{\mathbb M}}
\def\bbn{{\mathbb N}}
\def\bbo{{\mathbb O}}
\def\bbp{{\mathbb P}}
\def\bbq{{\mathbb Q}}
\def\bbr{{\mathbb R}}
\def\bbs{{\mathbb S}}
\def\bbt{{\mathbb T}}
\def\bbu{{\mathbb U}}
\def\bbv{{\mathbb V}}
\def\bbw{{\mathbb W}}
\def\bbx{{\mathbb X}}
\def\bby{{\mathbb Y}}
\def\bbz{{\mathbb Z}}

\def\Fa{\mathfrak a}
\def\Fb{\mathfrak b}
\def\Fc{\mathfrak c}
\def\Fd{\mathfrak d}
\def\Fe{\mathfrak e}
\def\Ff{\mathfrak f}
\def\Fg{\mathfrak g}
\def\Fh{\mathfrak h}
\def\Fi{\mathfrak i}
\def\Fj{\mathfrak j}
\def\Fk{\mathfrak k}
\def\Fl{\mathfrak l}
\def\Fm{\mathfrak m}
\def\Fn{\mathfrak n}
\def\Fo{\mathfrak o}
\def\Fp{\mathfrak p}
\def\Fq{\mathfrak q}
\def\Fr{\mathfrak r}
\def\Fs{\mathfrak s}
\def\Ft{\mathfrak t}
\def\Fu{\mathfrak u}
\def\Fv{\mathfrak v}
\def\Fw{\mathfrak w}
\def\Fx{\mathfrak x}
\def\Fy{\mathfrak y}
\def\Fz{\mathfrak z}

\def\fst{\frak{st}}

\def\FA{\mathfrak A}
\def\FB{\mathfrak B}
\def\FC{\mathfrak C}
\def\FD{\mathfrak D}
\def\FE{\mathfrak E}
\def\FF{\mathfrak F}
\def\FG{\mathfrak G}
\def\FH{\mathfrak H}
\def\FI{\mathfrak I}
\def\FJ{\mathfrak J}
\def\FK{\mathfrak K}
\def\FL{\mathfrak L}
\def\FM{\mathfrak M}
\def\FN{\mathfrak N}
\def\FO{\mathfrak O}
\def\FP{\mathfrak P}
\def\FQ{\mathfrak Q}
\def\FR{\mathfrak R}
\def\FS{\mathfrak S}
\def\FT{\mathfrak T}
\def\FU{\mathfrak U}
\def\FV{\mathfrak V}
\def\FW{\mathfrak W}
\def\FX{\mathfrak X}
\def\FY{\mathfrak Y}
\def\FZ{\mathfrak Z}

\def\scra{\mathscr{A}}
\def\scrb{\mathscr{B}}
\def\scrc{\mathscr{C}}
\def\scrd{\mathscr{D}}
\def\scre{\mathscr{E}}
\def\scrf{\mathscr{F}}
\def\scrg{\mathscr{G}}
\def\scrh{\mathscr{H}}
\def\scri{\mathscr{I}}
\def\scrj{\mathscr{J}}
\def\scrk{\mathscr{K}}
\def\scrl{\mathscr{L}}
\def\scrm{\mathscr{M}}
\def\scrn{\mathscr{N}}
\def\scro{\mathscr{O}}
\def\scrp{\mathscr{P}}
\def\scrq{\mathscr{Q}}
\def\scrs{\mathscr{S}}
\def\scrt{\mathscr{T}}
\def\scru{\mathscr{U}}
\def\scrv{\mathscr{V}}
\def\scrw{\mathscr{W}}
\def\scrx{\mathscr{X}}
\def\scry{\mathscr{Y}}
\def\scrz{\mathscr{Z}}

\newcommand{\rz}{\raisebox{.2ex}{*}}
\newcommand{\rzc}{\raisebox{.1ex}{\circ}}
\newcommand{\rzf}{\raisebox{.1ex}{f}}

\def\rns{\rz\bbr^n_{nes}}            
\def\rms{\rz\bbr^m_{nes}}            
\def\rps{\rz\bbr^m_{nes}}            
\def\rs{\rz\bbr_{nes}}            

\newcommand{\bsm}[1]{\boldsymbol{#1}}    

\newcommand{\crr}{$C^\infty(\bbr,\bbr)$}
\newcommand{\dcrr}{C^\infty(\bbr,\bbr)}
\newcommand{\cMn}{$C^\infty(M,\bbr^n)$}
\newcommand{\dcMn}{C^\infty(M,\bbr^n)}
\newcommand{\cM}{$C^\infty(M,\bbr )$}
\newcommand{\dcM}{C^\infty(M,\bbr )}
\newcommand{\cmn}{$C^\infty(\bbr^m,\bbr^n)$}
\newcommand{\cnn}{$C^\infty(\bbr^n,\bbr^n)$}
\newcommand{\dcmn}{C^\infty(\bbr^m,\bbr^n)}
\newcommand{\dcnn}{C^\infty(\bbr^n,\bbr^n)}
\newcommand{\cm}{$C^\infty(\bbr^m,\bbr)$}
\newcommand{\dcm}{C^\infty(\bbr^m,\bbr)}
\newcommand{\cn}{$C^\infty(\bbr^n,\bbr)$}
\newcommand{\dcn}{C^\infty(\bbr^n,\bbr)}
\newcommand{\cnpr}{$C^\infty_{pr}(\bbr^n,\bbr)$}
\newcommand{\cnnpr}{$C^\infty_{pr}(\bbr^n,\bbr^n)$}
\newcommand{\dcnpr}{C^\infty_{pr}(\bbr^n,\bbr)}
\newcommand{\dcnnpr}{C^\infty_{pr}(\bbr^n,\bbr^n)}

\newcommand{\strcrr}{$\rz C^\infty(\bbr,\bbr)$}
\newcommand{\dstrcrr}{\rz C^\infty(\bbr,\bbr)}
\newcommand{\strcmn}{$\rz C^\infty(\bbr^m,\bbr^n)$}
\newcommand{\dstrcmn}{\rz C^\infty(\bbr^m,\bbr^n)}
\newcommand{\strcm}{$\rz C^\infty(\bbr^m,\bbr)$}
\newcommand{\dstrcm}{\rz C^\infty(\bbr^m,\bbr)}
\newcommand{\strcn}{$\rz C^\infty(\bbr^n,\bbr)$}
\newcommand{\strcnn}{$\rz C^\infty(\bbr^n,\bbr^n)$}
\newcommand{\dstrcn}{\rz C^\infty(\bbr^n,\bbr)}
\newcommand{\dstrcnn}{\rz C^\infty(\bbr^n,\bbr^n)}
\newcommand{\strcnpr}{$\rz C^\infty_{pr}(\bbr^n,\bbr)$}
\newcommand{\dstrcnpr}{\rz C^\infty_{pr}(\bbr^n,\bbr)}
\newcommand{\strcMn}{$\rz C^\infty(M,\bbr^n)$}
\newcommand{\dstrcMn}{\rz C^\infty(M,\bbr^n)}
\newcommand{\strcM}{$\rz C^\infty(M,\bbr)$}
\newcommand{\dstrcM}{\rz C^\infty(M,\bbr)}

\newcommand{\scmn}{$SC^\infty(\bbr^m,\bbr^n)$}
\newcommand{\scnn}{$SC^\infty(\bbr^n,\bbr^n)$}
\newcommand{\scn}{$SC^\infty(\bbr^n,\bbr)$}
\newcommand{\scrr}{$SC^\infty(\bbr,\bbr)$}
\newcommand{\scnnpr}{$SC^\infty_{pr}(\bbr^n,\bbr^n)$}
\newcommand{\dscmn}{SC^\infty(\bbr^m,\bbr^n)}
\newcommand{\dscnn}{SC^\infty(\bbr^n,\bbr^n)}
\newcommand{\dscnnpr}{SC^\infty_{pr}(\bbr^n,\bbr^n)}
\newcommand{\scm}{$SC^\infty(\bbr^m,\bbr)$}
\newcommand{\dscm}{SC^\infty(\bbr^m,\bbr)}
\newcommand{\dscn}{SC^\infty(\bbr^n,\bbr)}
\newcommand{\dscrr}{SC^\infty(\bbr,\bbr)}
\newcommand{\scMn}{$SC^\infty(M,\bbr^n)$}
\newcommand{\dscMn}{SC^\infty(M,\bbr^n)}
\newcommand{\scM}{$SC^\infty(M,\bbr)$}
\newcommand{\dscM}{SC^\infty(M,\bbr)}

\newcommand{\cgcmn}{$^\s C^\infty(\bbr^m,\bbr^n)$}
\newcommand{\dcgcmn}{^\s C^\infty(\bbr^m,\bbr^n)}
\newcommand{\cgcMn}{$^\s C^\infty(M,\bbr^n)$}
\newcommand{\dcgcMn}{^\s C^\infty(M,\bbr^n)}
\newcommand{\cgcn}{${}^\s C^\infty(\bbr^n,\bbr)$}
\newcommand{\dcgcn}{{}^\s C^\infty(\bbr^n,\bbr)}
\newcommand{\cgcrr}{${}^\s C^\infty(\bbr,\bbr)$}
\newcommand{\dcgcrr}{{}^\s C^\infty(\bbr,\bbr)}
\newcommand{\tangm}{$C^\infty(T\bbr^m)$}

\newcommand{\tangM}{$C^\infty(TM)$}
\newcommand{\stangM}{$SC^\infty(TM)$}
\newcommand{\tangMnes}{$C^\infty(TM)_{\nes}$}
\newcommand{\strtangm}{$\rz C^\infty(T\bbr^m)$}
\newcommand{\strtangM}{$\rz C^\infty(TM)$}
\newcommand{\tangn}{$C^\infty(T\bbr^n)$}
\newcommand{\strtangn}{$\rz C^\infty(T\bbr^n)$}
\newcommand{\tanf}[1]{$C^\infty(#1^{-1}T\bbr^n)$}
\newcommand{\tanfnes}[1]{$C^\infty(#1^{-1}T\bbr^n)_{nes}$}
\newcommand{\strtanf}[1]{$\rz C^\infty(#1^{-1}T\bbr^n)$}
\newcommand{\dtangm}{C^\infty(T\bbr^m)}
\newcommand{\dtangM}{C^\infty(TM)}
\newcommand{\dstangM}{SC^\infty(TM)}
\newcommand{\dtangMnes}{C^\infty(TM)_{\nes}}
\newcommand{\dstrtangm}{\rz C^\infty(T\bbr^m)}
\newcommand{\dstrtangM}{\rz C^\infty(TM)}
\newcommand{\dtangn}{C^\infty(T\bbr^n)}
\newcommand{\dstrtangn}{\rz C^\infty(T\bbr^n)}
\newcommand{\dtanf}[1]{C^\infty(#1^{-1}T\bbr^n)}
\newcommand{\dtanfnes}[1]{C^\infty(#1^{-1}T\bbr^n)_{nes}}
\newcommand{\dstrtanf}[1]{\rz C^\infty(#1^{-1}T\bbr^n)}

\newcommand{\cts}[1]{\*C^\infty(T\bbr^#1)}

\newcommand{\RDM}{\SD_M^{\bbr}}
\newcommand{\lRDn}{\SD_n^{\bbr}}

\newcommand{\DM}{\SD_{M}}
\newcommand{\DMs}{\SD_{M,nes}}
\newcommand{\Dn}{\SD_{n}}
\newcommand{\Dns}{\SD_{n,nes}}
\newcommand{\flXt}{$\BF l^X_t$}           
\newcommand{\flXs}{$\BF l^X_s$}
\newcommand{\flYt}{$\BF l^Y_t$}
\newcommand{\flYs}{$\BF l^Y_s$}
\newcommand{\flXd}{$\BF l^X_{\d}$}
\newcommand{\flYd}{$\BF l^Y_{\d}$}
\newcommand{\dflXt}{\BF l^X_t}           
\newcommand{\dflXs}{\BF l^X_s}
\newcommand{\dflYt}{\BF l^Y_t}
\newcommand{\dflYs}{\BF l^Y_s}
\newcommand{\dflXd}{\BF l^X_{\d}}
\newcommand{\dflYd}{\BF l^Y_{\d}}

\newcommand{\norm}[1]{\left\|#1\right\|}
\newcommand{\normk}[1]{\left\|#1\right\|_k}
\newcommand{\normkk}[1]{\left\|#1\right\|_{k+1}}
\newcommand{\normkpl}[1]{\left\|#1\right\|_{k+1}}
\newcommand{\normo}[1]{\left\|#1\right\|_1}
\newcommand{\normz}[1]{\left\|#1\right\|_0}
\newcommand{\normj}[1]{\left\|#1\right\|_j}
\newcommand{\nsnormj}[1]{\rz \left\|#1\right\|_j}

\newcommand{\quo}[1]{\frac{\norm{#1}}{\norm{#1}+1}}
\newcommand{\quok}[1]{\frac{\normk#1}{\normk#1+1}}
\newcommand{\quoj}[1]{\frac{\normj#1}{\normj#1+1}}
\newcommand{\quoke}[1]{\frac{\normk#1}{\normkpl#1}}
\newcommand{\quoje}[1]{\frac{\normj#1}{\normj#1}}
\newcommand{\nsquoj}[1]{\frac{\rz \ngtn\normj#1}{*\ngtn\normj#1+1}}
\newcommand{\smk}[1]{\sum_{k=0}^{\infty}\frac{1}{2^k}#1}
\newcommand{\smj}[1]{\sum_{j=0}^{\infty}\frac{1}{2^j}#1}
\newcommand{\nssmj}[1]{\rz \ngtn\sum_{j=0}^{\infty}\frac{1}{2^j}#1}

\newcommand{\pow}[2]{\ensuremath{$#1^#2$}}

\newcommand{\st}{\scalebox{.8}[1.3]{$st$}}
\newcommand{\dsim}{\stackrel{\d}{\sim}}

\newcommand{\frwnd}{\mathbin{\text{\scalebox{1}[.7]{$\stackrel{\smallfrown}{\d}$}}}}

\def\k{\kappa}

\newcommand{\btrn}[1]{\bigtriangledown_#1}
\newcommand{\btrndg}{\bigtriangledown_{\d}(g)}
\newcommand{\trndg}{\raisebox{.1ex}{\bigtriangledown_{\d}(g)}}
\newcommand{\fpd}{\rz f\circ \varphi_{\d}}
\newcommand{\frad}{\frac{1}{\d}}
\newcommand{\pdf}{\psi_{\d} \circ \rz f}



\newtheorem{theorem}{Theorem}[section]
\newtheorem{lemma}{Lemma}[section]

\newtheorem{definition}{Definition}[section]
\newtheorem{corollary}{Corollary}[section]

\theoremstyle{definition}
\newtheorem*{remark}{Remark}

\title{Magnification Spaces: A nonstandard approach  to inverse mapping theorems
  }
\author{Tom McGaffey}
\date{}

\begin{document}
\maketitle
\tableofcontents
\begin{abstract} This paper develops an infinitesimal order of magnitude  coupled with overflow technique that allows nonnumerical proofs of nondegenerate and degenerate inverse mapping theorems for mappings minimally regular at a point.
                 This approach is used first to give a transparent proof of the inverse mapping theorem of Behrens and Nijenhuis and then is deployed to prove an inverse mapping result for mappings whose linear part vanishes at the differentiable point. We finish by indicating further possible capacities of this approach.
\end{abstract}

\section{Summary:  Setting and strategy}
\subsection{Summary and strategy}
    In this paper we will develop a nonstandard framework for differential calculus that comes close to allowing one to ``throw away higher order terms'' (see Stroyan's paper \cite{Stroyan1977} for a nonstandard take on this goal) in considerations of how `regularity' of the top term of the Taylor expansion of differentiable map implies regularity of the function. To do this we will first  develop  rudiments of a theory of magnification spaces and their morphisms sufficient for the purposes of this paper. Although some of the constructions used here have appeared in a less functorial form, see Stroyan and Luxemburg, \cite{StrLux76}, an essentially functorial development of infinitesimal almost linear mappings and a natural embedding of differentiable objects will expose a natural scale invariant morphism structure. Further, this structure will be such that the grading determined by vanishing order will {\it in this infinitesimal setting} fall into sharp relief, quite unlike the behavior of mappings on the standard level.    In particular, a variety of inverse mapping theorems will become natural consequences of magnification space characterizations of differentiability, a transparent infinitesimal regularity and a natural use of overflow. More specifically, the Behrens-Nijenhuis inverse mapping theorem (for mappings uniformly differentiable at a point) will be an easy overflow consequence of a transparent proof of an infinitesimal version of this result. These tools, even in this elementary state of development, will analogously yield a singular inverse mapping theorem, where the differential vanishes at the differentiable point. 

    We will assume familiarity with at least the rudiments of nonstandard methods on Euclidean spaces. Higher saturations are not used so that eg., an  enlargement of just lower levels of a superstructure are assumed. The first chapter of Lindstr{\o}m, \cite{Lindstrom1988}, the introduction by Henson, \cite{Henson1997}, followed by Cutland's chapter, \cite{CutlandNSrealAnaly1997},  or  even the author's introduction to the subject in his dissertation, \cite{McGaffeyPhD} are quite sufficient.

\subsection{Differentiability and magnification mappings}
    The transfer of a mathematical object $A,f,r$ will generally be prefixed with a * superscript, $\rz A,\rz f,\rz r$, although we will often leave off the * if its setting is clear from context. General internal objects will often, but not always, be written in Latin, Greek or Fraktur.
    We give a mere introduction of the  magnification space framework that is part of our unfinished  work on a nonstandard approach to the stability of mappings as in Guillemin and Golubitsky's classic text, \cite{GolubitskyGuillemin1973}.
    We will be working strictly in finite dimensional real Euclidean spaces, generally $\bbr^m$ or $\bbr^n$ for $m,n\in\bbn$. $\bbr_+$ will denote the strictly positive real numbers, $\rz\bbr_+$ its transfer, etc. $\mu(0)$ will denote the set of infinitesimals in $\rz\bbr_+$ and $\mu(0)_+=\{\Fr\in\mu(0):\Fr>0\}$.
    For $x\in\bbr^m$, let $B_\d(x)=B^m_\d(x)=\{y\bbr^m:|y-x|<\d\}$, let $B_d=B_d(0)$, $\rz B_d(x)$ its transfer and also for $\Fd\in\rz\bbr_+$ and $\xi\in\rz\bbr^m$, $\rz B_\Fd(\xi)$ the analogous *open ball of radius $\Fd$ at $\xi$.

    If  $\d\in\mu(0)_+$, then $D^m_\d=D_\d(0)$ will denote the $\rz\bbr_{nes}$ module $\bsm{\rz\bbr^m_\d}=\d\rz\bbr^m_{nes}=\{v\in\rz\bbr^m:|v|/\d\in\rs\}$, and in general, if $\z\in\rz\bbr_{nes}^m$, $D^m_\d(x)$ will denote $\{x+\xi:\xi\in D^m_\d\}$. As the dimension of the space we are working in will be implicit, we will often write $D_\d$ for $D^m_\d$, etc., and also we will generally consider the situation where $\z$ is standard, ie., $\rz x$ for some $x\in\bbr^m$ and leave off the *.
    Let $\bsm{Q_x}=Q^\d_x=\{\f{\xi-x}{\d}:\xi\in D_\d(x)\}$ with the restricted norm. Let $\bsm{\flat_x}=\flat^{m,\d}_x:D_\d(x)\ra
    Q_x$ be the map $\xi\mapsto \f{\xi-x}{\d}$.  The map $\flat_x$ clearly is the $\d$ dilation map at $y$.
    In particular, we have that $Q_x$ (with the restricted norm) is just $\rz\bbr^m_{nes}$; eg.,
    we have a \textbf{well defined standard part map $\bsm{\fst_x}:Q_x\ra\bbr^m$} and
    hence \textbf{the $\bbr$ linear injection $\bsm{\s_x}:\bbr^m\ra Q_x:v\mapsto\rz v$ is well defined.}
    Denote the image by $\bsm{^\s Q_x}$ to indicate that these are the standard points in $Q_x$.  Note that $\fst_x\circ\s_x$ is the identity on $\bbr^m$. We will generally suppress the subscript on $\fst_x$, ie., just write $\fst$ noting that the space we are acting on will usually be clear.

    Next, note that \textbf{there is a canonical identification of the image of $\bsm{Q^\d_x}$ under the standard part map with $\bsm{T_x\bbr^m}$.} That is, we have the setup
     \begin{align}
       Q^\d_x\stackrel{\fst}{\ra}T_x\bbr^m\stackrel{\s}{\ra}Q^\d_x.
     \end{align}
    We will shortly see that this canonical sequence is even functorial for maps that are differentiable; ie., the respective maps preserve this triple.
    Note that the above actually gives a canonical identification of the set (standard) tangent vectors at $x$ with an $\bbr$ subspace of the $\rz\bbr_{nes}$ module $Q^\d_x$. In fact, these correspond to those vectors in $D_\d(x)$ of the form $x+\d v$ for $v\in\bbr^m$.  In particular, all vectors  $v\in Q^\d_x$ are infinitesimally close to a tangent vector, namely $\fst v$. In a later paper, we will see that the cloud of vectors in $Q^\d_x$ infinitesimally close to a given tangent vector in $T_x\bbr^m$ will function as the set of {\it almost derivations}.

    With respect to maps, note that although continuous maps $f:\bbr^m\ra\bbr^n$ preserve monads, they do not preserve the $D_\d$'s. It is elementary that for $f$ to be differentiable at $x$, it is necessary  that $f(D^m_\d(x))\subset D^n_\d(f(x))$ for $\d\in\mu(0)_+$. (Actually, the nonstandard condition is equivalent to this condition for a sufficiently full subset of infinitesimal $\d$'s.)
    There are a variety of infinitesimal conditions that are equivalent to differentiability of $f$ at $x$.
    Nonetheless, the above expression implies that any differentiable map $f:(\bbr^m,x)\ra(\bbr^n,f(x))$ induces, via the bijective dilation maps $\flat_x$ and $\flat_{f(x)}$, the map
     \begin{align}
       f^\d_x\;\dot=\;\flat_{f(x)}\circ\rz f\circ(\flat_x)^{-1}:(Q^\d_x,0)\ra (Q^\d_{f(x)},0).
     \end{align}
    In the following, one should see the extensive discussion in Stroyan and Luxemburg, \cite{StrLux76}, p94-110. In particular, one sees on p94 the difference (from a nonstandard perspective) between differentiability and uniform differentiability and on p95 a definition of almost linearity.

\begin{definition}\label{def:pointwise, almost linear map}
    Suppose that $\Ff:\rz\bbr^m_{nes}\ra\rz\bbr^n_{nes}$ is internal. This will mean that $\Ff$ is defined on an internal set containing $\rz\bbr^m_{nes}$ with range an internal set containing $\rz\bbr^n_{nes}$ with $\Ff(\rz\bbr^m_{nes})\subset\rz\bbr^n_{nes}$. Then we say
    that $\bsm{\Ff}$ is almost linear if for all
    $\a,\b\in\rz\bbr_{nes}$ and all $v,w\in\rz\bbr^m_{nes}$,
    $\Ff(\a v+\b w)\sim \a \Ff(v)+\b \Ff(w)$.
    We say that $\Ff$ is S-continuous on $\rz\bbr_{nes}^m$ if given $v, w\in\rz\bbr^m_{nes}$ with $v\sim w$, we have $\Ff(v)\sim\Ff(w)$. (This is a slight variation on the usual notion.)
\end{definition}
    It's clear that if $\Ff$ is almost linear, then an infinitesimal perturbation $\Fg$ of $\Ff$, ie., with $\Ff(\xi)-\Fg(\xi)\sim 0$ for all $\xi\in\rz\bbr^m_{nes}$, is also almost linear.
    Let's give some simple facts on differentiable functions and almost linear internal functions.
\begin{lemma}
    If mapping $f:\bbr^m\ra\bbr^n$ is differentiable at $x$, then $f^\d_x:Q_x^\d\ra Q^\d_{f(x)}$ is almost linear for all $\d\in\mu(0)_+$.
    An internal $\Ff:\rz\bbr^m_{nes}\ra\rz\bbr^n_{nes}$ is almost linear if and only if
     ${}^o\Ff$ is linear and $\fst_n\circ \Ff=\;^o\!\Ff\circ \fst_m$ where
    $\fst_m:\rz\bbr^m_{nes}\ra\bbr^m$, is the standard part map $\fst_m(v)$ sometimes denoted $^o\!v$.
    Suppose $d\in\bbr_+$ and $\Ff:\rz\bbr^n_{nes}\ra\rz\bbr^n_{nes}$ is almost linear. If $\Ff(\rz\bbr^n_{nes})\supset \rz B_{d}$, then ${}^o\Ff$ is a linear isomorphism.
\end{lemma}
\begin{proof}
    The verification of the first assertion can be picked out of the text on pages 94 and 95 in \cite{StrLux76}.
    It's easy to see that almost linearity implies $S$ continuity on $\rz\bbr_{nes}^m$, ie., if $v,w\in\rz\bbr^m_{nes}$ with $v\sim w$, then $\Ff(v)\sim\Ff(w)$ by writing in terms of a basis and this immediately implies $\fst\circ\Ff={}^o\Ff\circ\fst$. Linearity of ${}^o\Ff$ is clear.
    Conversely, the commutativity condition is equivalent to ${}^o\Ff\circ\fst\sim\Ff$ on $\rz\bbr^m_{nes}$, Therefore, for $\a,\b\in\rz\bbr_{nes}$ with $a={}^o\a$, $b={}^o\b$ and $\Fv,\Fw\in\rz\bbr^m_{nes}$ with $v={}^o\Fv$ and $w={}^o\Fw$, we have $\Ff(\a\Fv+\b\Fw)\sim{}^o\Ff(\fst(\a\Fv+\b\Fw))={}^o\Ff(av+bw)=a{}^o\Ff(v)+b{}^o\Ff(w)\sim\a\Ff(\Fv)+\b\Ff(\Fw)$.
    Let $b\in B_d$, so that there is $\xi\in\rz\bbr^m_{nes}$ with $\Ff(\xi)=\rz b$, then ${}^o\Ff(\fst\xi)=\fst\Ff(\xi)=\fst(\rz b)=b$, ie., ${}^o\Ff:\bbr^n\ra\bbr^n$ is a linear map whose image contains $B_d$.
\end{proof}

    The following corollary (hopefully portraying the functorality of our setup) contains information relating our objects and mappings along with material that will give a new twist to proving differentiability of the inverse map in the theorem below.
\begin{corollary}\label{cor: f bar is S continuous}
    Suppose that $f:(\bbr^m,x)\ra(\bbr^n,y)$ is differentiable at  $x$.
    Then for $\d\in\mu(0)_+$ we have the following commutative diagram.
      \begin{align}\label{diagram: df ladder}
        \begin{CD}
          T_x\bbr^m                  @>df_x>>                 T_y\bbr^n \\
          @AA\text{st} A                                       @AA\text{st} A   \\
          \SQ_x^{\d,m}                  @>f^\d_x>>                 \SQ_y^{\d,n}  \\
          @AA \flat_x A                                         @AA \flat_y A   \\
          D_\d(x)                       @>f>>                    D_\d(y)  \\
        \end{CD}
      \end{align}

    In particular, for all $\d\in\mu(0)_+$, $f^\d_x:\SQ_x^{\d,m}\ra\SQ_y^{\d,n}$ is almost linear and for $\d_1,\d_2\in\mu(0)_+$ and $^of^{\d_1}_x=\;^of^{\d_2}_x$.
    Conversely, if there exists a linear map $L:\bbr^m\ra\bbr^n$ such
    that for all $\d\in\mu(0)_+$, $f^\d_x$
    exists, is almost linear and $^o(f_x^\d)=L$, then $f$
    is differentiable at $x$.
\end{corollary}

\begin{proof}

    The bottom square is commutative by definition.
    The commutativity of the top square follows immediately from the previous lemma.
    Conversely, we must prove the following standard
    statement. There exists linear $L:\bbr^m\ra\bbr^n$ such that the following holds.
    For each $b>0$, there exists $\bar{c}>0$ such that
    if   $v\in\bbr^m$ with $|v|= 1$  and $0<c<\bar{c}$ then  $|\f{1}{c}(f(x+cv)-f(x))-L(v)|<b$.
    Fix the given positive $b\in\bbr$. By hypothesis, we have that  $0<\d\sim
    0$ implies $^o(f^\d_x)(v)=L(v)$ for $v\in\rz\bbr^m_{nes}$.
    This implies, in particular, that if $|v|\leq 1$, $|f^\d_x(v)-L(v)|\sim 0$,
    and so certainly $|f^\d_x(v)-L(v)|<b$, ie., $|\f{1}{\d}(f(x+\d
    v)-f(x))-L(v)|<b$. Let $\Gamma$ denote
\begin{align}\label{eqn: set where diff quotient less than b}
    \Big\{\bar{\d}\in\rz\bbr_+:\Big|\f{\rz f(x+\d v)-\rz f(x)}{\d}-\rz L(v)\Big|<\rz b\;\text{if}\;|v|=1\;\text{and}\;0<\d<\bar{\d}\Big\}.
\end{align}
    $\Gamma$ is an internal set which by above contains $\{\bar{\d}\in\rz\bbr:0<\bar{\d}\sim
    0\}$, and so contains a $\bar{\d}\not\sim 0$. Let
    $\bar{c}=\f{^o\bar{b}}{2}<\bar{\d}$, eg., $\rz\bar{c}\in\Gamma$.
    Hence if $v\in\bbr^m$ with $|v|=1$ and $c\in\bbr^m$ with
    $0<c<\bar{c}$, it follows that
    $|\f{1}{c}(f(x+cv)-f(x))-L(v)|<b$, as we needed.
\end{proof}
\begin{remark}
    Note that an alternative approach to the last part of the above proof is to note that $\G=\rz G$ where
     \begin{align}
       G=\{\ov{r}\in\bbr_+:\Big|\f{f(x+rv)-f(x)}{r}-L(v)\Big|<b,\;\text{if}\;|v|=1\;\text{and}\;0<r<\ov{r}\}
     \end{align}
    so that any $\ov{\d}\in\rz G$ implies that $G$ is nonempty. The same observation applies to our overflow proof of the inverse mapping theorem. In particular, we don't have to assume that our regularity holds for all $\d\in\mu(0)_+$ for this indeed to be true! One can observe that this remark does not apply to our second order theorem in the following section.
\end{remark}
    \textbf{The vast majority of the constructions in this paper will occur in the magnification spaces, $Q^\d_x$, for reasons we believe will become apparent.} We believe that these spaces, due to the functorality that we have begun to demonstrate above and due to their nice order of magnitude scaling of differentiable functions, should be a quite useful tool in the study of regularity properties of differentiable objects. In the final section, among other things, we will indicate how they can be defined on differentiable manifolds.


\section{Behrens-Nijenhuis Inverse Mapping Theorem}\label{sec: BehrensNijen IFT}

    In this part, we will give our new proof of the  inverse function theorem of Behrens, \cite{Behrens1974}, and Nijenhuis, \cite{NijenhuisInverseMapping1974} (and also see analogous result of Knight, \cite{KnightStrngIFT1988}).
    We see this approach to infinitesimal regularity implies local regularity  as prototypical for a mode of argument that can see general use. In the next part, section \ref{sec: 2nd order IFT}, we will use this magnification/overflow argument in the framework of the higher magnification spaces to prove a more general inverse mapping theorem. Here, we begin with Behrens original definition of uniform differentiability and a magnification space reformulation.
\begin{definition}
    Suppose that $f:\bbr^m\ra\bbr^n$ and $x\in\bbr^m$. We say that
    $f$ is \textbf{uniformly differentiable at $x$} if there exists a linear
    map $L:\bbr^m\ra\bbr^n$, its derivative at $x$, such that the following holds.
    For each $r\in\bbr_+$, there is $s\in\bbr_+$ such that if $y,z\in B_s(x)$, then
      \begin{align}
        |f(y)-f(z)-L(y-z)|<r|y-z|
     \end{align}
    This conditions are equivalent to the following infinitesimal criterion. If $\xi,\z\in\mu(x)$, then
      \begin{align}
        \rz f(\xi)-\rz f(\z)-\rz L(\xi-\z)=o(\xi-\z).
      \end{align}
\end{definition}

    An important hint at the following is that although equivalent formulations, the error term of the first may be a positive standard  fraction of the linear approximation; not so in the nonstandard formulation where the error term is infinitely smaller than the linear term. But as the linear term of the nonstandard formulation is also infinitesimal, the degree of the strategic improvement seems uncertain. Yet when seen in terms of magnification spaces,. the improvement becomes clear as the next two lemmas demonstrate.

\begin{lemma}\label{lem: f unif diff imply f^d_x is unif diff}
    Suppose that $f:\bbr^m\ra\bbr^n$ is uniformly differentiable at $x$ with with first order linear term $L$ and suppose that $0<\d\sim 0$, $\xi,\z\in D_\d(x)$ and $\ov{\xi}=\flat_x^\d(\xi),\ov{\z}=\flat_x^\d(\z)$ are the corresponding elements of $\SQ^\d_x$. Then
\begin{align}
    f^\d_x(\ov{\xi})-f^\d_x(\ov{\z})=L(\ov{\xi}-\ov{\z})+\ov{\nu}(\ov{\xi}-\ov{\z})
\end{align}
    where $\ov{\nu}:Q^{\d,m}_x\ra\rz\bbr$ is an internal map satisfying $\ov{\nu}(\ov{\la})\ll |\ov{\la}|$.
\end{lemma}
\begin{proof}
    The proof amounts to unfolding the definitions. first note that, by definition, $\xi=x+\d\ov{\xi}$ and $\z=x+\d\ov{\z}$ are the elements of $D_\d(x)$ defining $\ov{\xi}$ and $\ov{\z}$ respectively. Next, uniform differentiability implies ($\bsm{\diamondsuit}$) $f(x+\d\ov{\xi})=f(x+\d\ov{\z})+L(\d(\ov{\xi}-\ov{\z}))+o(\d(\ov{\xi}-\ov{\z}))$. But recall that $f^\d_x(\ov{\xi})$ is defined by the relation $f(x+\d\ov{\xi})=f(x)+\d f^\d_x(\ov{\xi})$ and similarly for $f^\d_x(\ov{\z})$. Substituting these into ($\diamondsuit$) we have the expression $f(x)+\d f^d_x(\ov{\xi})=f(x)+\d f^\d_x(\ov{\z})+\d L(\ov{\xi}-\ov{\z})+\d o(\ov{\xi}-\ov{\z})$ and the result follows. (Note that we used $\d o(v)=o(\d v)$.)
\end{proof}
\begin{remark}
    {\it Let's expand on the critical part of the last lemma. Before we dilate the behavior of $f$ at $x$, the best we can hope for is that the magnitude of the error term is a small standard fraction of the magnitude of the linear term. After $f$ has been  dilated at $x$, ie., when modeled as $f^\d_x$ on the magnification space $Q^\d_x$, the linear term is unchanged, eg., of noninfinitesimal magnitude, but now the error term is infinitesimal. This improvement now allows a crude argument to prove this infinitesimal version of the theorem. Because of the invariance of this argument  with  respect to the $\d$ scale, overflow then allows us to pushes this into the standard world. Let us further note that in the world of scientific computing, it seems that a version of this fact has been known for more than a generation ago, see Berz \cite{BerzAutomaticDifferentiation} for an explicit discussion and \cite{ShamBerz2010} for some dramatic examples.}
\end{remark}
    Given this, we will prove a lemma that will give the neighborhood surjection part of the inverse mapping theorem. The result can be significantly sharpened, but as stated is sufficient for the proof of the following theorem. This lemma functions as the component of the proof of IFT where we must prove the map is open; this is the part of the usual proof that uses some machinery. On the other hand, our lemma is easily proved due to the greatly improved approximation as noted in the above remark.
    If $\SA\subset\rz\bbr^n$ is internal, we will let $\ov{\SA}$ denote the *closure of $\SA$ in the transfer of the Euclidean topology.

\begin{lemma}\label{lem: onto result for pf of IFT}
    Suppose that internal $\nu:\rz\bbr^n\ra\rz\bbr^n$ is *continuous with $|\nu(x)|\ll |x|$ for $x\in \rz B_d\;\dot=\{x\in\rz\bbr^n:|x|\leq d\}$  and  in $d\in\bbr_+$ is arbitrary and fixed.
    Suppose an internal $\Ff:\rz\bbr^m\ra\rz\bbr^n$ is defined such that $\Ff(\xi)-\Ff(\z)=\;\xi-\z+\nu(\xi-\z)$, eg., is almost linear. Then $\rz B_{d/2}\subset \Ff(\rz B_d)$,
    and $\Ff^{-1}$ exists on $\rz B_{d/2}$ and is almost linear.
\end{lemma}
\begin{proof}
    First note that it suffices to prove $\Ff(\rz B_d)\cap \rz B_{d/2}$ is *dense in $\rz B_{d/2}$ as *transfer implies that $\Ff(\rz B_d)$ is *closed and so $\rz B_{d/2}\subset\ov{\Ff(\rz B_d)}=\Ff(\rz B_d)$.
    Suppose that $\SS\;\dot=\;\Ff(\rz B_d)\cap\rz B_{d/2}$ is not *dense in $\rz B_{d/2}$. Then, by *compactness of $\SS$, there is $\Fw_0\in\rz B_{d/2}$ and $\Fv_0\in \SS$ such that if $\z=\Fw_0-\Ff(\Fv_0)$, then $0<\Fr_0\dot=|\z|=\rz dist(\SS,\Fw_0)$.
    So $\Ff(\Fv_0)+\z=\Fw_0$ and therefore uniform differentiability implies that $\Ff(\Fv_0+\z)=\Ff(\Fv_0)+\z+\nu(\z)=\Fw_0+\nu(\z)$, ie., $\Ff(\Fv_0+\z)-\Fw_0=\nu(\z)$, a contradiction as $\nu(\z)=o(\z)$
    Furthermore,  $|\nu(x)|\ll|x|$ implies that $2|\xi-\z|>|\Ff(\xi)-\Ff(\z)|>|\xi-\z|/2$ which  clearly implies $\Ff$ is injective and $\Ff^{-1}$ is S-continuous from which almost linearity follows by expanding $\Ff^{-1}\circ\Ff(\a\Fv+\b\Fw)$ and using S-continuity of $\Ff^{-1}$.
\end{proof}

    The previous two lemmas, giving elementary properties of uniformly differentiable functions will allow an intuitive proof of the inverse mapping theorem.
    {\it Essentially, we have verified the infinitesimal inverse mapping theorem in the spaces $Q^\d_x$ and below we see how a routine overflow gets the standard theorem.}
\begin{theorem}[Local Inverse Function Theorem]\label{thm: inverse fcn thm}
    Suppose that $f:\bbr^m\ra\bbr^m$ is uniformly differentiable at $x$ and that $df_x:T_x\bbr^m\ra T_y\bbr^m$ is a
    linear isomorphism. Then $f$ is a diffeomorphism at $x$; ie., a
    homeomorphism onto a neighborhood of $y$ and its inverse is
    differentiable at $y$ with differential $(df_x)^{-1}$.
\end{theorem}
\begin{proof}
     In this part of the proof, by precomposing with $df_x^{-1}$, we can assume that $df_x$ is the identity map.   Let $0<\d\sim 0$ and note that by Lemma \ref{lem: f unif diff imply f^d_x is unif diff}, we have that $f^\d_x:\SQ^\d_x\ra\SQ^{\d}_x$ is uniformly differentiable with linear term the identity. Then using Lemma \ref{lem: onto result for pf of IFT}, we have that there is $0<d\in\bbr$, independent of $\d$, such that if $B^\d_d\subset\SQ^\d_x$ denotes the ball of radius $d$ centered at the origin, we have\;\;  $(1_\d)\!:\!f^\d_x(B^\d_d)\supset B^\d_{d/2}$,\; $(2_\d)\!:\!f^\d_x$ is $1-1$ on $B^\d_d$\; and\; $(3_\d)\!:\!|(f^\d_x)^{-1}(\la)-(f^\d_x)^{-1}(\eta)|\geq |\la-\eta|/2$ for $\la,\eta\in f^\d_x(B^\d_d)$. The statements $(1_\d),\;(2_\d)\;\text{and}\;(3_\d)$ are internal and hold for each $0<\d\sim 0$. Therefore, by overflow we have that there is a standard $b>0$ such that $(1_b),(2_b)$ and $(3_b)$ all hold.
     But, unraveling the definitions, recalling that the maps $\flat_\d$ are bijections for $0<\d\in\rz\bbr$ and standard motions if $\d\in\bbr$, we have that $(1_b)$ says that
\begin{align}
     \rz f(x+b\rz\!B_d)= \rz f(x)+b\rz\!f^b_x(\rz B_d)\notag \\
     \supset \rz f(x)+\rz\!B_{bd/2}\;\dot=\rz V,
\end{align}
     and use reverse transfer. $(2_b)$ implies that $f$ is $1-1$ on $V$ and the estimate $(3_b)$  implies that the standard map $(f^b_x)^{-1}$ is continuous on $V$.
     Finally, if $y=f(x)$,  lemma \ref{lem: onto result for pf of IFT} gives $(f^{-1})^\d_y$ is almost linear with $^o((f^{-1})^\d_y)$ the identity for all $\d\in\mu(0)_+$, so that corollary \ref{cor: f bar is S continuous} implies that $f^{-1}$ is differentiable at $y$.

\end{proof}
\begin{remark}
     Note that  beyond the use of overflow, this argument is a gathering of simple infinitesimal information already exposed.
     In particular, a systematic use of the elementary machinery associated with the magnification spaces, eg., diagram \ref{diagram: df ladder} and lemma \ref{lem: f unif diff imply f^d_x is unif diff}, allows one to prove differentiability of $f^{-1}$ {\it without calculations}.   Further, all other proofs of this result (except the nonstandard infinitesimal recursion proof of Cutland and Hanqiao, \cite{CutlandIFT1993}, and their proof is of the continuously differentiable implicit function theorem)  use careful estimates and rely on a version of some contraction mapping theorem or an argument borrowed from its proof, especially when proving that the map is open. Our proof used crude order of magnitude arguments and a basic functorial machinery associated with magnification spaces.  The magnification spaces and overflow make these much simpler arguments effective. Part of the purpose of this paper is to expose the possibilities of this approach.
     Also critical to this process is the two fold use of infinitesimal scaling. First, we blow up an infinitesimal scale to work within the given $\SQ^\d_x$, but then we use the infinitesimal scaling within the given fixed $\SQ^\d_x$ to establish the regularity properties there. This approach will be used with a bit more effort in the following section for the proof of the second order (singular) inverse mapping theorem although the machinery there is not as polished as the first order material in this section.
\end{remark}

\section{A singular second order inverse mapping theorem}\label{sec: 2nd order IFT}
\subsection{Perspective}
    In the previous subsection, we used the following critical facts to give a transparent proof of the inverse mapping theorem. First: when viewing our mapping $f$ as a mapping $f^\d_x$ on $\SQ^d_x$, we found that the term following the linear term was infinitely smaller than the linear term on all of $\SQ^\d_x$, so that any regularity carried by the linear term was not disturbed by the remainder term; hence easy proofs that $f^\d_x$ had the regularity properties of the linear term. Second: this construction was independent of $\d$, eg., holding for all positive infinitesimal $\d$, and for each such $\d$, the statement of this regularity is an internal statement. This means that the set of $\d$'s where the regularity statement holds is internal and contains the set of positive infinitesimals. Unwinding the consequences of overflow then finishes the proof of the standard inverse mapping theorem.

    For the inverse mapping theorem, regularity of the (first nontrivial) linear term means  linear isomorphism and so the inverse mapping theorem says that this regularity property of the linear term determines the regularity of our map $f$ at $x$; ie.,  $f$ is ``(locally) like a linear isomorphism'', ie. a diffeomorphism. But if the first order term vanishes, what is regularity with respect to the second order term? The best we can hope for the quadratic term is an appropriately differentiable homeomorphism near our point. Of course, our uniform differentiability hypothesis above  makes no assumptions  about regularity away from our point $x$ and similarly the assumption below, differentiability to order 2 at $x$ makes no such assumptions. So the best we can hope for our map $f$ (now assumed to be differentiable to degree 2, see definition \ref{def: unif differentiable to degree k} below) is a (differentiable) homeomorphism from some neighborhood of $x$ to a neighborhood of $f(x)$. This is what we will prove below.
    Note that in the corollary following our theorem the results is placed in the context of map genericity, see \cite{GolubitskyGuillemin1973}.

    Historically, there have been a variety of inverse mapping theorems (and implicit mapping theorems, which we don't discuss; again see Krantz and Parks, \cite{KrantzParks1992}). Recent work that greatly generalize the classical continuously differentiable theorem  mostly  follow the insights of Clarke, \cite{Clarke1976}, generally evolving out of the notion of the generalized Jacobian. As with the notion of generalized Jacobian, all of these approaches relate the approximate linear behavior (by generic linear maps) of a mapping at the germ of a point in its domain; see eg., Jetakumar and Luc, \cite{UnbddGenerzJacobians-Jeyakumar-Luc2002}, for recent work along this line. (A notable exception is the work of Frankowska, \cite{FrankowskaInvMapThmss1990} and her coworkers.) Our approach (as in the previous section) instead makes negligible the behaviors of higher order parts with respect to  lower order parts {\it on infinitesimal scales}. Hence, their approach, although quite powerful, is limited to linearization perspectives. Our second result gives a quadratic regularity result that does not seem to be proveable  within the aegis of their machinery. (Note that the significant work of Fukui, Kurdyka and Paunescu, \cite{FukuiTameIFT2010}, works within the framework of tame mappings, eg., their conditions contain restrictions for a neighborhood of a point.) In the conclusion we indicate how the approach here is only an introduction to a quite general nonstandard approach to such questions.
\subsection{A singular second order inverse mapping theorem}
    A useful setting for our singular inverse mapping theorems is the set of differentiable maps of order $k$ at a given point $x\in \bbr^n$. Obviously differentiability  of order 1 at a point is distinctively weaker than uniform differentiability. The fact that we will be using differentiability of degree greater than 1  gives us a uniform smoothness at a point analogous to uniform differentiability. Yet, as noted above, we will be assuming that the derivative is as singular (as a linear map) as possible at $x$, ie., assuming that the derivative vanishes.

    We will give a (nonstandard) infinitesimal definition as it will be sufficient for our purposes. Here, if $v\in\bbr^n$, and $j\in\bbn$, $v^{\odot j}$ will denote the $j^{th}$ symmetric power of $v$, ie., $v\odot\cdots\odot v\in\odot^j\bbr^n$. (Our reference for the multilinear algebra is the text of Greub, \cite{GreubMultilinear1978}.) If we have a positive definite norm $|v|$ on $\bbr^n$, then one can check that this norm induces a norm  on $\odot^2\bbr^n$ via the definition $|v\odot w|=|v||w|$. We will use the transfer of these structures without mention and as in the following definition will often leave off *'s.
\begin{definition}\label{def: unif differentiable to degree k}
    If $f:\bbr^n\ra\bbr^n$ and $k\in\bbn$, we say that $f$ is differentiable to degree $k$ at $x$ if there is $L\in End(\bbr^n)$ and $S^j\in Hom(\odot^j\bbr^n,\bbr^n)$ for $j=2,\cdots,k$, such that if $\d\in\mu(0)_+$ and $v\in\rz\bbr^n$ with $|v|=1$, we have
\begin{align}
    f(x+\d v)=f(x)+\d L(v)+\d^2 S^2(v^{\circ 2})+\cdots+\d^k S^k(v^{\circ k})+\Fr^\d_k(v)\\
    \text{where}\quad \Fr_k^\d(v)=o(\d^k)\qquad\qquad\qquad\qquad\qquad\qquad\notag
\end{align}
    Below we will often use the notation $d^2f_x$ for the the quadratic term $S^2$, calling it the quadratic differential.
\end{definition}
     Note that we have absorbed the usual factorials into the definitions of the symmetric multilinear maps $S^j$.
\begin{lemma}\label{lem: remainder term is *continuous}
    If $f:(\bbr^n,x)\ra\bbr^n$ is differentiable to degree $k$, $k\in\bbn$, at $x$, with expansion given above, then the remainder term $\Fr_k^\d$ is *continuous at $x$.
\end{lemma}
\begin{proof}
    This is straightforward: it's easy to see that $f$ is differentiable implies that $f$ is continuous at $x$ and so  $\rz f$ is *continuous at $\rz x$. But the $S^j$'s are in fact analytic and so their *dilations are *analytic and so *continuous.
\end{proof}

    We begin with a definition of the geometric regularity condition our quadratic parts need to satisfy. Following the theorem, we show how this condition has a natural genericity expression in the sense of \cite{GolubitskyGuillemin1973}.
\begin{definition}
    Let $a,\ov{d},c\in\bbr_+$ be arbitrary and fixed. Suppose that $\SH:\rz\bbr^n\ra\rz\bbr^n$ is an internal map. Then we say that $\SH$ is $(c,\ov{d})$ quadratically regular at $0$ for $\d\in\mu(0)_+$ if $\SH$ is *continuous with $\SH(\rz\!B_d(\xi))\supset\d\rz\!B_{cd^2}(\SH(\xi))$ for all $d\in\bbr_+$ with $d<\ov{d}$ and $|\xi|<\rz a$
\end{definition}
     The following lemma will play the same role in the proof of the higher order  singular inverse function theorem that Lemma \ref{lem: onto result for pf of IFT} played in the proof  the inverse function theorem. Note that it appears that the  needed result (for this quadratic case) for a standard proof that should correspond to the contraction theorem argument for the regular inverse mapping theorem does not seem to exist for this singular quadratic situation. Unlike the method used here which is used for both the inverse mapping theorem of the previous section and the the singular theorem presented here; it seems a very different procedure will be needed for a standard proof of this singular case.
     Nevertheless, we believe that our proof of the following lemma is more complicated than it needs to be.
\begin{lemma}\label{lem: H quad homeo implies H+o(H) also}
    Let $a,\ov{d},c\in\bbr$ be positive numbers and $\d\in\mu(0)_+$. Suppose that $\SH:\rz\bbr^n\ra\rz\bbr^n$ is internal, *continuous, injective and $(c,\ov{d})$ quadratically regular for $\d$ at $0$.
    Suppose that $\Ff:\rz\bbr^n\ra\rz\bbr^n$ is internal and satisfies $\Ff(\xi)-\Ff(\z)= H(\xi)-H(\z)+\nu(H(\xi)-H(\z))$ where $\nu$ is internal, *continuous and satisfies $\nu(\la)=o(|\la|)$. Given this, if $0<d\in\bbr$ with $d<\ov{a}$, then we have that $\Ff|\rz B_d$ is one to one and $\Ff(\rz B_d)\supset\d\rz B_{cd^2/2}$.
\end{lemma}
\begin{proof}
    Let's first verify the surjection assertion. As $\Ff$ is *continuous and $\rz B_d$ is *closed, it suffices to show that $\d\rz B_{cd^2/2}\cap\Ff(\rz B_d)$ is *dense in $\d\rz B_{cd^2/2}$. If not, then, there is $\Fw_0\in\rz B_{c\d^2/2}$ and  $0<\Fr_0\in\rz\bbr$  such that $\rz dist(\Ff(\rz B_d),\Fw_0)=\rho_0$, and so by *continuity of $\Ff$ and as everything is *closed, there is $\Fx_0\in\rz B_d$ with $\rz dist(\Ff(\Fx_0),\Fw_0)=\rho_0$. First note that $|\rho_0|=o(\d d^2)$ as $H(\rz B_d)\supset\d\rz B_{cd^2}$ and $\Ff(\xi)=H(\xi)+o(H(\xi))$. But this implies that $\Fx_0\in\rz Int(\rz B_d)$, the *interior of $\rz B_d$. Otherwise, $|\Ff(\Fx_0)|\geq |H(\Fx_0)|-o(H(\Fx_0))\geq \f{4}{5}\d cd^2$ (say) which would prevent $\rz dist (\Ff(\Fx_0),\Fw_0)$ to be $o(\d d^2)$.
    Therefore, for some $0<\e\sim 0$, we have $\rz B_\e(\Fx_0)\subset \rz B_d$. Now, by hypothesis, $\d\rz B_{cd^2}(0)\subset H(\rz B_\e(\Fx_0))-H(\Fx_0)$ and so
\begin{align}
     (\diamondsuit)\quad  \d\rz B_{cd^2}(\Ff(\Fx_0))\subset\Ff(\Fx_0)+H(\rz B_\e(\Fx_0))-H(\Fx_0).\notag
\end{align}
    But there is $\ov{\Fz}\in\d\rz B_{cd^2}(\Ff(\Fx_0))$ such that, for example
\begin{align}
    \rz dist(\ov{\Fz},\Fw_0)<\rho_0-\f{\d cd^2}{3}.
\end{align}
     And by $(\diamondsuit)$ there is a $\ov{\Fy}\in\rz B_\e(\Fx_0)$ such that $\Ff(\Fx_0)+H(\ov{\Fy})-H(\Fx_0)=\ov{\Fz}$ and so
\begin{align}
    \rz dist(\Ff(\Fx_0)+H(\ov{\Fy})-H(\Fx_0),\Fw_0)+\f{\d cd^2}{3}<\rz dist(\Ff(\Fx_0),\Fw_0).
\end{align}
    Now $\nu(H(\ov{\Fy})-H(\Fx_0))=o(\d cd^2)$ and so  $\d cd^2/3-|\nu(H(\ov{\Fy})-H(\Fx_0))|>c\e^2/4$;
    and so by the hypothesis on $\Ff$
\begin{align}
    \rz dist(\Ff(\ov{\Fy}),\Fw_0)=\rz dist(\Ff(\Fx_0)+H(\ov{\Fy})-H(\Fx_0)+\nu(H(\ov{\Fy})-H(\Fx_0)),\Fw_0) \notag\\
    <\rho_0-(\f{\d cd^2}{3}-|\nu(H(\ov{\Fy})-H(\Fx_0))|)<\rho_0-\f{\d cd^2}{4},\qquad
\end{align}
    contradicting $\rz dist(\Ff(\rz B_d),\Fw_0)=\rho_0$.
    Finally note that its clear that $\Ff$ is one to one as $\xi\not=\z$ implies that $H(\xi)-H(\z)\not=0$ and $\nu(H(\xi)-H(\z))=o(H(\xi)-H(\z))$.
\end{proof}
    We now have the preliminaries required to prove the next result.
\begin{theorem}\label{thm: quad reg -> homeomorph}
    Suppose that $f:(\bbr^n,x)\ra\bbr^n$ is differentiable to order 2 at $x$ and that $df_x=0$. Suppose that there are $c,d\in\bbr_+$ such that for $\d\in\mu(0)_+$ and $\rz\!B_d=\{\xi\in Q^\d_x:|\xi|<d\}$, we have that $\xi\in\rz\!B_d\mapsto\d d^2f_x(\xi,\xi)$ is $(c,d)$ quadratically regular for $\d$.
    Then $f$ is a homeomorphism from a neighborhood of $x$ onto a neighborhood of $f(x)$.
\end{theorem}
\begin{proof}
    Suppose that $0<\d\sim 0$ is arbitrary and let $d\in\bbr$ be positive to be determined shortly. Now $f$ is uniformly differentiable to second order at $x$ means that we have for $\xi\in\rz\bbr^n$ with $|\xi|\leq d$ that $f(x+\d\xi)=f(x)+\d df_x(\xi)+\d^2d^2f_x(\xi,\xi)+o(\d^2d^2f_x(\xi,\xi))$ and so writing a similar expression for $f(x+\d\z)$ for $|\z|\leq d$ also; and as $df_x=0$, we get with some manipulation that
\begin{align}\label{eqn: f^d_x difference formula in d^2f}
    f^\d_x(\xi)-f^\d_x(\z)=\d(d^2f_x(\xi,\xi)-d^2f_x(\z,\z))+\nu(\d(d^2f_x(\xi,\xi)-d^2f_x(\z,\z)))
\end{align}
     where by the hypothesis $\nu(\la)=o(|\la|)$ for $\la\in\SQ^\d_{f(x)}$. Here we have written $\xi,\z\in \rz B_d(0)\subset Q^\d_x$ (the *transfer of the identification of $ T_x\bbr^n$ with the corresponding vectors in $\bbr^n$ being implicit).
    Now Lemma \ref{lem: remainder term is *continuous} implies that the remainder `$\nu$' term is *continuous.
    The hypothesis says that, for the given $c,d\in\bbr_+$,  $\xi\mapsto \SH(\xi)=\d d^2f_x(\xi,\xi):\SQ^\d_x\ra\SQ^\d_{f(x)}$ satisfies the hypothesis of Lemma \ref{lem: H quad homeo implies H+o(H) also} for a given arbitrary $\d\in\mu(0)_+$.
    So we can apply lemma \ref{lem: H quad homeo implies H+o(H) also}, where  $\Ff(\xi)$ in that lemma is  $f^\d_x(\xi)$ satisfying expression \ref{eqn: f^d_x difference formula in d^2f} above.  That is, for our given $c,d>0$ in $\bbr_+$, we have a  *neighborhood $U^\d$ of $0$ in $\SQ^\d_{f(x)}$ such that the following holds:
\begin{align}
    \quad f^\d_x:\rz B^\d_d(0)\twoheadrightarrow U^\d\supset\d\rz B^{\d}_{cd^2}\;\text{is a *homeomorphism onto}\;U^\d \tag{$expr_\delta$}\label{eqn: eqn delta}
\end{align}
    Now \ref{eqn: eqn delta} is an internal statement for each fixed $\d$ and so the set $\FI=\{\d\in\rz\bbr_+:expr_\d\;\text{holds}\}$ is an internal set and by the above argument we know that $\mu_+(0)\subset\FI$, and so eg., $\FI$ contains a standard positive number $a$. That is, we have that $f^a_x$  maps $\rz B^a_d$ *homeomorphically (and hence homeomorphically as $f^a_x=\rz f^a_x$ is standard when $a,x$ is standard) onto a neighborhood $U^a\supset a\rz B^a_{cd^2}$
\end{proof}


\subsection{Relationship with transversality to Segre variety}
    It turns out that our quadratic regularity condition on $d^2f_x$ can be expressed as a genericity condition (in the spirit of \cite{GolubitskyGuillemin1973}) with respect to the Segre variety. This interpretation will be developed here.
    We begin with lemmas that situate the vanishing of the
    differential in the framework of the magnification spaces.
    Now $d^2f_x$ is a symmetric bilinear map and is used above in the form $v\mapsto d^2f_x(v,v)$.
    Hence, it is natural to consider a factorization of this operator via the following.
\begin{definition}
    Let $p_2:V\ra V^{\odot 2}$ denote the map $v\mapsto v^{\odot 2}$. Let $\SS=\SS_2\subset V^{\odot 2}$ denote the image of $p_2$.
\end{definition}
    The map $p_2$ will play the role as a universal factor for our singular maps that are second order regular. Note that $dp_{2,0}:T_0V\ra T_0V^{\odot 2}$ is the zero map, but we have the following statement, of which all but the last part is commonly known, see Greub, \cite{GreubMultilinear1978}.
\begin{lemma}\label{lem:p^2 is quad homeo}
     $\SS_2$ is a smooth $n$ dimensional submanifold of $V^{\odot 2}$ and  $p_2$ is a homeomorphism of $V$ onto $\SS_2$ that is a diffeomorphism away from $0$. Specifically, if $0<d\leq 1$, then there is  $0<r,c\in\bbr$, with $r<d$, such that if $B_r(v)\subset B_d(0)$, then $p_2(B_r(v))\supset B_{cr^2}(p_2(v))$.
\end{lemma}
\begin{proof}
     We need to verify the last statement. If $v=0$, the statement is clear; assume $v\not=0$. Given this, note that the map $v+tw\mapsto (v_0+tw)^{\odot 2}$ maps an $r$ ball on centered on $v_0$ into $B_{r^2}(p_2(v))\cap\SS_2$ and the differential of this map is just $dp_{2,v_0}:T_{v_0}V\ra T_{p_2(v_0)}\SS_2$ which satisfies $dp_{2,v_0}(w)=\f{d}{dt}(v_0+tw)^{\odot 2}|_{t=0}=2v_0\odot w$, ie., is an isomorphism as $v_0\not=0$, ie.,   and in fact if $|v|=r\leq 1$, $|dp_{2,v}(w)|=2r|w|$, so that the image of the ball $B_s(0)\subset T_{v_0}V$ under $dp_{2,v_0}$ contains the ball $B_{rs}(0)\subset T_{p_2(v_0)}V^{\odot 2}$.
     But, *transferring this and noting that for $0<\d\sim 0$, we have that $p_2(v_0+\d w)=p_2(v_0)+\d dp_{2,v_0}(w)+o(\d)$ so that $dp_{2,v_0}(B_\d(0))\supset B_{r\d}(0)$ along  with an argument similar to that in Lemma \ref{lem: onto result for pf of IFT}, gets that $p_2(B_\d(v_0))\supset B_{\d s/2}(p^2(v_0))$. Overflow with respect to the parameter $\d$ then gets this inclusion out to some standard value $\ov{r}$ and choosing $r=\min\{\ov{r},s\}$ gets our conclusion.
\end{proof}

     In the considerations for our singular inverse mapping theorem we will considering  the relationship of the quadratic differential $d^2f_0$ acting on $V^{\odot 2}$ which we will denote by $\wt{d^2f_0}$, ie. as $d^2f_0$ is symmetric bilinear,  we are defining, as is typically done, eg., see [Greub], the linear map $\wt{d^2f_0}:V^{\odot 2}\ra V$ by $\wt{d^2f_0}(v\odot w)=d^2f_0(v,w)$. (Note here that as we are in a vector space setting, we are identifying tangent vectors at some point in $V$ with vectors in $V$, and so identifying $T_vV^{\odot 2}$ with $V^{\odot 2}$.)
     Let $\SK_2(f)<V^{\odot 2}$ denote the kernel of the linear map $\wt{d^2f_0}$. At this point we need to recall the notion of transversality. For an excellent treatment, see \cite{GolubitskyGuillemin1973}, our definition will be taylored to this context. Suppose that $M,P$ are finite dimensional smooth manifolds without boundary, $g:M\ra P$ is a smooth map and $j:N\hookrightarrow P$ is the embedding map for a smooth submanifold $N$ of $P$.
\begin{definition}
      We say that $g$ is transversal to $N$ if $g(M)\cap N$ is empty for $dim(M)<codim(N)$, or if for each $y=f(x)\in g(M)\cap N$, we have that
       \begin{align}
          dim(dg_x(T_xM)/dg_x(T_xM)\cap T_yN)=dim(P)-dim(N).
       \end{align}
    This is written $g\ov{\pitchfork} N$. Intuitively,  $dg(TM)$ must fill out the `normal space' to $N$ at intersection points.
\end{definition}
    Note that, again see \cite{GolubitskyGuillemin1973}, the fundamental theorem on transversality is the Thom Transversality Theorem (of which there are several versions). For us, this result says that, in say the $C^k$ topology, $k\in\bbn$ to be determined, on smooth maps $g:M\ra P$, the set of $g$ that satisfy $g\ov{\pitchfork}N$ are open and dense.  Our situation has a twist in the sense that the map corresponding to $g$ will be the canonical embedding of $\SS$ and the submanifold $N$ (that is tacitly considered fixed) will be the kernel subspace $\SK_2(f)$ of $\wt{d^2f_0}$.
    That is, if $i_\SS:\SS\hookrightarrow V^{\odot 2}$ is the inclusion map. Then, in the case that $df_x=0$, we will consider the (generic) situation where $i_\SS\;\ov{\pitchfork}\;\SK_2(f)$. At this point, we need the following lemma.
\begin{lemma}\label{lem: f transversal implies H quad homeo}
    Suppose that $f:(V,0)\ra (V,0)$ is differentiable to order 2 at $0$ with $df_0=0$, and suppose that $i_\SS\;\ov{\pitchfork}\;\SK_2(f)$ (at $0$). Then the map $v\mapsto H(v)\;\dot=\;d^2f_0(v,v)=\wt{d^2f_0}\circ p^2(v)$ is a quadratic homeomorphism at $0$. That is, there are positive $b,c\in\bbr$ such that we have that $H|B_b(0)$ is a homeomorphism onto a neighborhood of $0$ and for all $0<\d\sim 0$ we have $H(\rz B_\d(0))\supset \rz B_{c\d^2}(0)$.
    In particular, for any $\d\in\mu(0)_+$, the internal map $\SH:\rz V\ra\rz V$ given by $\xi\mapsto\d\rz\!H(\xi)$ is $(c,d)$ quadratically regular for $\d$.
\end{lemma}
\begin{proof}
    First of all, we will see that $i_\SS\;\ov{\pitchfork}\;\SK_2(f)$ will imply that there is $0<\ov{c}\in\bbr$ such that if $\SB_{\ov{c}}\subset V^{\odot 2}$ denotes the metric ball in $V^{\odot 2}$ of radius $\ov{c}$, then $\wt{d^2f_0}|\SB_{\ov{c}}\cap\SS$ is a diffeomorphism from $\SB_{\ov{c}}\cap\SK_2(f)$ to a neighborhood of $0$ in $V$. The inverse mapping theorem will imply this last assertion once one verifies that for $u\in\SS$ in some neighborhood of $0$ in $V^{\odot 2}$ we have that $d(\wt{d^2f}_0|\SS)_u$ is an isomorphism from $T_u\SS$ to $T_wV$ where $w=\wt{d^2f}(u)$. To simplify notation, let $\SA\dot=\wt{d^2f}_0$, the linear map from $V^{\cdot 2}$ to $V$.
    Given this, first note that as $\SA\circ i_\SS$ is smooth, then $\SA|\SS_2$ is a smooth map. But transversality along with a dimension count  implies that $Im d(i_\SS)_0\cap ker d\SA_0=\{0\}$. At this point note that as we are working in the vector space $V$ and as  $\SA$ is linear, if $v\in V^{\odot 2}$, then writing  $\SE_v\dot=ker d\SA_v$, we have $\SE_v$ is just a parallel translate (by $v$) of the subspace $\SE_0$. That is, translation by $v\in V^{\odot 2}$, $\Ft_v:T_0V^{\odot 2}\ra T_vV^{\odot 2}$ is a (canonical) linear isomorphism such $v\mapsto \Ft_v$ is  continuous (even smooth, but this is not needed) and $\Ft_{-v}(\SE_v)=\SE_0$. But its a special case of a well known fact about differentiable submanifolds of a finite dimensional vector spaces that the map $\Fg:v\in \SS\mapsto \Ft_{-v}(T_v\SS)< T_0 V^{\odot 2}$ is a continuous map to the Grassmann manifold of $k=dim \SS$ dimensional subspaces of $T_0V^{\odot 2}$. Given that the set of $k$ dimensional subspaces $W$ of $T_0V^{\odot 2}$ satisfying $W\cap\SE_0=\{0\}$ is an open set $U$  (in this Grassmannian), the continuity of $\Fg$ implies that $\wt{U}\dot=\Fg^{-1}(U\cap Im\Fg)$ is open in $\SS$.
    But this just says that $ker d\SA_v|T\SS_v=\{0\}$ for all $v\in\wt{U}$.
    So that we have that $H$ is the composition of a diffeomorphism and, by Lemma \ref{lem:p^2 is quad homeo}, a quadratic homeomorphism which is clearly a quadratic homeomorphism. (Of course, we will need to choose a new $c>0$ to compensate for the linear distortions of the diffeomorphism.)
\end{proof}
    The previous work immediately implies the following genericity result for those `regular' maps among those with vanishing derivative. One might compare this result with that of Phein, \cite{PhienLipIFT2012}, noting that our inverse mappings are generally not Lipschitz and also noting the relationship of transversality with stability of regularity under perturbation.
\begin{corollary}
    Suppose that $f:(\bbr^n,x)\ra\bbr^n$ is differentiable to order 2 at $x$ and that $df_x=0$. Suppose that $i_{\SS^2}\;\ov{\pitchfork}\;ker(\wt{d^2f_x})$. Then $f$ is a homeomorphism from a neighborhood of $x$ onto a neighborhood of $f(x)$.
\end{corollary}
\begin{proof}
    This is a consequence of the previous result, lemma \ref{lem: f transversal implies H quad homeo} and theorem \ref{thm: quad reg -> homeomorph}.
\end{proof}

\section{Perspective: Higher order magnification spaces}
   First of all, the reader should see that at this point the analogous higher order results will hold via a similar argument as the second order. Some preliminary calculation are given below, but, more importantly from the author's perspective, a general framework can be developed for mappings vanishing whose `lower order terms' vanish along subspaces (or more general algebraic subsets); this will be developed in later papers. Furthermore, it seems that by restricting these functions to, eg. a locally embedded copy of the Levi-Civita field (see eg., Todorov and Wolf, \cite{TodorovWolfHahnFieldRep}),  we may have much more control over infinitesimals tests for local invertibility. On a historical note: it's not inconceivable that Levi-Civita had something like this in mind, see Laugwitz, \cite{LaugwitzCantorVeronese2002}, eg., p119.

     Let's just give an idea of how to iterate the magnification space framework.
     Just as we introduced the magnification $\rs$ modules ${}Q^\d_x$
     of a $\d$ neighborhood of $x\in\bbr^m$, we can similarly magnify the almost
     $\rs$ submodules of ${}Q^\d_x$. For example, if $\xi\in {}Q^\d_{x}$, then we have $D_\e(\xi)\subset Q^\d_\xi$ and so can define ${}Q^\e_\xi({}Q^\d_x)$ as the set $\{\f{\z-\xi}{\e}:\z\in D_\e(\xi)\}$. Note that this set is canonically isomorphic to $\rz\bbr^m_{nes}$ and that for $\xi\in{}^\s Q^\d_x$, ie., a standard point in $Q^\d_x$, $Q_\xi^\e(Q^\d_x)$ has well defined (second order?) standard points.
       For the second order expansion, we get
     \begin{align}
        f(x+\d\xi+\d^2\rho)=f(x)+\d f^\d_x(\xi)+\d^2 (f^\d_x)^\d_\xi(\rho)
     \end{align}
   and similarly we get a third order expansion
     \begin{align}
        f(x+\d\xi+\d^2\rho+\d^3\la)&=\\
        f(x)&+\d f^\d_x(\xi)+\d^2(f^\d_x)^\d_\xi(\rho)+\d^3((f^\d_x)^\d_\xi)^\d_\rho(\la).\notag
     \end{align}

    It seems clear at this point (although the author hasn't written down all of the details) that higher order versions of the singular inverse mapping theorem of the previous section follow from  arguments analogous to the proof of the quadratic theorem.
    For now, we will just write down the next higher order version of the formula \ref{eqn: f^d_x difference formula in d^2f}. For $\xi,\xi'\in Q^\d_0(Q^\d_0)$, we have
     \begin{align*}
       (f^\d_x)^\d_0(\xi)-(f^\d_x)^\d_0(\xi')=\hspace{2in}\\
       \d^2\big(d^3f_x(\xi,\xi,\xi)-d^3f_x(\xi',\xi',\xi')\big)+\d^2o\big(d^3f_x(\xi,\xi,\xi)-d^3f_x(\xi',\xi',\xi')\big)
     \end{align*}
     This has the same general form as formula \ref{eqn: f^d_x difference formula in d^2f}, and possibly the reader can see that a similar argument will work. This material will appear in later work.

     Let's end with some general remarks. We believe that it's important to point out that this mode of orders of magnitude expansion is much more versatile than a Taylor expansion. For example, unlike the Taylor expansion, this expansion allows us to, at each successive expansion order, shift infinitesimally the center of expansion. Furthermore, this expansion is independent of choice of coordinates and actually exists on the  nested sequence of infinitesimal balls. As we will see in a later paper, this sequence exists and is well defined on sufficiently differentiable manifolds, $M$, with metric (and independent of the choice of such). Note, in this, that we will see that almost linear structures exist on the $Q^\d_x\subset\rz M$.
     Further, it's conceivable that such expansion with more general filtered orders of magnitude than integral powers of an infinitesimal could be useful.

\bibliographystyle{amsplain}
\bibliography{nsabooks}

\providecommand{\bysame}{\leavevmode\hbox to3em{\hrulefill}\thinspace}
\providecommand{\MR}{\relax\ifhmode\unskip\space\fi MR }
\providecommand{\MRhref}[2]{%
  \href{http://www.ams.org/mathscinet-getitem?mr=#1}{#2}
}
\providecommand{\href}[2]{#2}
\begin{thebibliography}{10}

\bibitem{Behrens1974}
M.~Behrens, \emph{{A local inverse function theorem}}, {Victoria Symposium in
  Nonstandard Analysis} (A.~Hurd and P.~Loeb, eds.), Contemporary Mathematics,
  vol. 302, Springer Verlag, 1974, pp.~34--36.

\bibitem{BerzAutomaticDifferentiation}
Martin Berz, \emph{{Automatic differentiation as nonarchimedean analysis.}},
  {Proceedings of the 3rd International IMACS-GAMM Symposium on Computer
  Arithmetic and Scientific Computing (SCAN-91)}, {Amsterdam: North-Holland},
  1992.

\bibitem{Clarke1976}
F.~H. Clarke, \emph{{On the inverse function theorem}}, Pacific J. Math.
  \textbf{64} (1976), 97--102.

\bibitem{CutlandNSrealAnaly1997}
Nigel~J. Cutland, \emph{Nonstandard real analysis}, {Nonstandard Analysis
  ({E}dinburgh, 1996)}, NATO Adv. Sci. Inst. Ser. C Math. Phys. Sci., vol. 493,
  Kluwer Acad. Publ., Dordrecht, 1997, pp.~51--76. \MR{1603229 (99b:26043)}

\bibitem{CutlandIFT1993}
Nigel~J. Cutland and Han~Qiao Feng, \emph{An infinitesimal proof of the
  implicit function theorem}, Glasgow Math. J. \textbf{35} (1993), no.~2,
  163--166. \MR{1220558 (94d:26021)}

\bibitem{FrankowskaInvMapThmss1990}
H{\'e}l{\`e}ne Frankowska, \emph{Some inverse mapping theorems}, Ann. Inst. H.
  Poincar\'e Anal. Non Lin\'eaire \textbf{7} (1990), no.~3, 183--234.
  \MR{1065873 (91j:49020)}

\bibitem{FukuiTameIFT2010}
Toshizumi Fukui, Krzysztof Kurdyka, and Laurentiu Paunescu, \emph{Tame
  nonsmooth inverse mapping theorems}, SIAM J. on Optimization \textbf{20}
  (2010), no.~3, 1573--1590.

\bibitem{GolubitskyGuillemin1973}
Martin Golubitsky and Victor Guillemin, \emph{{Stable Mappings and Their
  Singularities}}, Springer-Verlag, 1973.

\bibitem{GreubMultilinear1978}
W.H. Greub, \emph{{Multilinear Algebra}}, Hoch Schultexte/Universitexts,
  Springer-Verlag, 1978.

\bibitem{Henson1997}
C.~Ward Henson, \emph{{Foundations of nonstandard analysis: A gentle
  introduction to nonstandard extensions}}, {Nonstandard Analysis and its
  Applications} (L.~Arkeryd, N.~Cutland, and C.~W. Henson, eds.), Kluwer
  Academic, 1997, pp.~1--49.

\bibitem{UnbddGenerzJacobians-Jeyakumar-Luc2002}
V.~Jeyakumar and D.~T. Luc, \emph{An open mapping theorem using unbounded
  generalized {J}acobians}, Nonlinear Anal. \textbf{50} (2002), no.~5, Ser. A:
  Theory Methods, 647--663. \MR{1910909 (2003d:49027)}

\bibitem{KnightStrngIFT1988}
William~J. Knight, \emph{{A strong inverse function theorem}}, Am. Math.
  Monthly \textbf{95} (1988), 648--651.

\bibitem{KrantzParks1992}
Steven~G. Krantz and Harold~R. Parks, \emph{{The Implicit Function Theorem}},
  Birkh\"auser Boston Inc., Boston, MA, 2002, History, theory, and
  applications. \MR{1894435 (2003f:26001)}

\bibitem{LaugwitzCantorVeronese2002}
Detlef Laugwitz, \emph{Debates about infinity in mathematics around 1890: the
  {C}antor-{V}eronese controversy, its origins and its outcome}, NTM (N.S.)
  \textbf{10} (2002), no.~2, 102--126. \MR{1907246 (2003b:01027)}

\bibitem{Lindstrom1988}
Tom Lindstr{\o}m, \emph{{An invitation to nonstandard analysis}}, {Nonstandard
  Analysis and its Applications} (N.~Cutland, ed.), Cambridge University Press,
  1988, pp.~1--105.

\bibitem{McGaffeyPhD}
T.~McGaffey, \emph{{Regularity and nearness theorems for families of local Lie
  groups}}, Ph.D. thesis, {Rice University}, 2011.

\bibitem{NijenhuisInverseMapping1974}
Albert Nijenhuis, \emph{Strong derivatives and inverse mappings}, Amer. Math.
  Monthly \textbf{81} (1974), 969--980. \MR{0360958 (50 \#13405)}

\bibitem{PhienLipIFT2012}
P.~{Phien}, \emph{{Some quantitative results on Lipschitz inverse and implicit
  functions theorems}}, ArXiv e-prints (2012).

\bibitem{ShamBerz2010}
K.~Shamseddine and M.~Berz, \emph{{Analysis on the Levi-Civita field, a brief
  overview}}, {Advances in p-Adic and Non-Archimedean Analysis} (M.~Berz and
  K.~Shamseddine, eds.), Contemporary Mathematics, vol. 508, American
  Mathematical Society, 2010, pp.~215--237.

\bibitem{Stroyan1977}
K.~Stroyan, \emph{{Infinitesimal analysis of curves and surfaces}}, {Handbook
  of Mathematical Logic} (K.~J. Barwise, ed.), North Holland, 1977,
  pp.~197--231.

\bibitem{StrLux76}
K.D. Stroyan and W.A.J. Luxemburg, \emph{{Introduction to the Theory of
  Infinitesimals}}, Academic Press, 1976.

\bibitem{TodorovWolfHahnFieldRep}
Todor Todorov and Robert Wolf, \emph{{Hahn field representation of {A}.
  Robinson's asymptotic numbers}}, January~30 2006, Comment: 18 pages,
  pp.~357--374.

\end{thebibliography}

\end{document}